\documentclass[natbib]{svjour3}

\title{Natural Deduction for the Sheffer Stroke and Peirce's Arrow
  (And Any Other Truth-Functional Connective)}
\titlerunning{Natural Deduction for the Sheffer Stroke}

\author{Richard Zach\thanks{Published in: \emph{Journal of Philosophical Logic} (2016) 45:183--197. \href{https://doi.org/10.1007/s10992-015-9370-x}{DOI 10.1007/s10992-015-9370-x}}}
\institute{Richard Zach \at University of Calgary, Department of Philosophy, 2500 University Dr NW, Calgary, AB T2N0A9, Canada, 
http://ucalgary.ca/rzach/, \email{rzach@ucalgary.ca}}

\usepackage[colorlinks,linkcolor=black,citecolor=blue,urlcolor=blue,pdftex]{hyperref}
\usepackage{proof,lplfitch,amssymb,mathptmx}

\smartqed

\renewcommand{\fitchctx}[1]{%
\advance \fitchctxwidth by -\fitchsep%
\advance \fitchctxwidth by .5pt%
\begin{tabular}[b]{r@{}|p{\fitchctxwidth}@{}l}
\phantom{\slider} \\[-1.75ex]
#1 \\[-1.75ex] & &
\end{tabular}
\advance \fitchctxwidth by \fitchsep%
\advance \fitchctxwidth by -.5pt%
}

\newcommand{\subproofx}[2]{%
\advance \fitchprfwidth by -\fitchsep%
\advance \fitchprfwidth by .5pt%
\hspace*{.35em}%
\begin{tabular}[t]{|p{0pt}@{}p{\fitchprfwidth}@{\hspace*{\fitchsep}}l}
  \multicolumn{3}{@{}l@{}}{\ }\\[-2.35ex]
  #1 \\
  \ \\[-2.5ex] \cline{1-1}\\[-2ex]
  #2 
 \end{tabular}
\advance \fitchprfwidth by \fitchsep%
}

\def\possessivecite#1{\citeauthor{#1}'s \citeyearpar{#1}}
\def\shef{\mathbin{|}}
\def\fesh{\mathbin{\|}}
\def\pier{\mathbin{\downarrow}}
\def\xor{\oplus}

\def\LS{\mathbf{LS}}
\def\LJ{\mathbf{LJ}}
\def\LK{\mathbf{LK}}
\def\NJ{\mathbf{NJ}}
\def\NK{\mathbf{NK}}
\def\NS{\mathbf{NS}}
\def\formula#1{\ensuremath{#1}}

\begin{document}
\maketitle

\begin{abstract}
Methods available for the axiomatization of arbitrary finite-valued
logics can be applied to obtain sound and complete intelim rules for
all truth-functional connectives of classical logic including the
Sheffer stroke (\textsc{nand}) and Peirce's arrow (\textsc{nor}).  The
restriction to a single conclusion in standard systems of natural
deduction requires the introduction of additional rules to make the
resulting systems complete; these rules are nevertheless still simple
and correspond straightforwardly to the classical absurdity
rule. Omitting these rules results in systems for intuitionistic
versions of the connectives in question.
\end{abstract}

\section{Introduction}

In a recent paper, \cite{HazenPelletier:14} compared
\possessivecite{Jaskowski:34} and \possessivecite{Gentzen:34}
versions of natural deduction.  In Section~2.3 of their paper, Hazen
and Pelletier consider \possessivecite{Price:61} natural deduction
rules for the Sheffer stroke (\textsc{nand}).  These rules are:
\[
\infer[\shef I_P]{A \shef B}{\infer*{B \shef B}{[A]}}
\qquad
\infer[\shef E_P]{B \shef B}{A & A \shef B}
\qquad
\infer[{\shef}{\shef} E_P]{B}{A & (B \shef B) \shef A}
\]
Although these rules are sound and complete for a logic in which
$\shef$ is the only primitive, they are not of the standard
``intelim'' form.  Hazen and Pelletier attempt to give a
Ja\'skowski-Fitch style natural deduction system in Section~3.3 which
comes closer to this ideal but find that their rules do not
characterize the classical logic of~$\shef$ except in the attenuated
sense that adding these rules to a standard natural deduction system
allows one to prove $(A \shef B) \leftrightarrow \lnot(A \land B)$
(``parasitic completeness'').  By themselves, their rules are
``ambiguous'' between the two intuitionistic versions $\lnot(A \land
B)$ and $\lnot A \lor \lnot B$.  They then discuss the framework of
Schr\"oder-Heister's generalized natural deduction and show that Do\v
sen's $\star$, which is an ``indigenous Sheffer function'' for
intuitionistic logic in the sense that every connective can be
expressed in terms of it, can be given intelim rules in this
framework.  They leave open the question of how the classical Sheffer
stroke should be dealt with in natural deduction, as well as the
status of their proposed rules. The aim of this note is to show how
one can give a natural deduction system of standard intelim rules for
the classical Sheffer stroke and its intuitionistic variants, and to
discuss the properties of the resulting systems. It is of course
possible to arrive at the same rules directly; indeed, \cite{Read1999}
did so. The methods presented here, however, are completely general
and can be applied to any set of truth-functional connectives; we
sketch this also for Pierce's arrow (\textsc{nor}) and exclusive
or.

\section{Obtaining natural deduction rules}

In a series of papers, \cite{BaazFermZach:93,BaazFermZach:94,Zach:93}
showed how finite-valued logics can be given sequent calculus
axiomatizations based on a simple procedure that yields introduction
rules for each connective and truth value, and where sequents have as
many places as the logic has truth values. They proved completeness,
cut-elimination, midsequent theorem, interpolation, and generalized it
to multiple-conclusion natural deduction systems for finite-valued
logics. The natural deduction systems have the normal form property,
and there are translations from natural deduction derivations to
sequent calculus proofs and vice versa. In the case of two-valued
classical logic, the construction yields exactly the classical sequent
calculus $\LK$, and a multiple-conclusion system of natural deduction
with the standard introduction rules for~$\lnot$, $\land$, $\lor$,
$\to$, $\forall$, and $\exists$, and ``general elimination rules''.
This system is complete for classical logic without the need to add
axioms (e.g., $A \lor \lnot A$) or rules such as double negation
elimination or Prawitz's $\bot_C$ to Gentzen's $\NJ$.\footnote{The
  general elimination rules coincide with the elimination rules of
  \possessivecite{Parigot:92a} free deduction; he showed how to obtain
  the standard elimination rules from them by systematic
  simplification. A version of the simplified system of classical
  multiple-conclusion natural deduction was studied
  by~\citet{Parigot:92}. The general elimination rules for $\land$ and
  $\to$ were studied by \citet{vonPlato2001}.}

For the Sheffer stroke, this procedure yields sequent calculus rules
as follows.  The Sheffer stroke is a truth-functional connective, and
any truth function can be expressed as a conjunctive normal form of
its arguments.  We can thus express the conditions under which a
formula of the form $A \shef B$ is true in terms of a conjunction of
disjunctions of $A$, $\lnot A$, $B$, and~$\lnot B$.  In the same way
we can express the conditions under which $A \shef B$ is false, i.e.,
express $\lnot(A \shef B)$ as a conjunctive normal form.
\begin{description}
\item[T] $A \shef B$ is true iff $A$ is false or $B$ is false:
\[A \shef B \Leftrightarrow \lnot A \lor \lnot B\]
\item[F] $A \shef B$ is false iff $A$ is true and $B$ is true:
\[\lnot(A \shef B) \Leftrightarrow A \land B\]
\end{description}

From these normal forms we then directly obtain introduction rules
for~$\shef$ in the succedent and antecedent.  The normal form for $A
\shef B$ provides the introduction rule for $\shef$ in the
succedent~$\shef$R; that for $\lnot(A \shef B)$ the introduction rule
for the antecedent~$\shef$L.  In each case, the conjuncts in the
normal form correspond to the premises.  In each premise sequent, put
the formula $A$ on the right if it occurs without negation in the
corresponding conjunct, and on the left if it occurs negated, and the
same for~$B$.  Thus, for~$\shef R$ we have one premise since the
conjunctive normal form has just one conjunct $\lnot A \lor \lnot B$,
and since $A$ and $B$ are both negated, $A$ and $B$ both go on the
left: $A, B \vdash$.  For~$\shef L$ we have two premises corresponding
the the two (one-element) disjunctions $A$ and~$B$: $\vdash A$ and
$\vdash B$.  The sequents in the full rule then also contain
the context formulas $\Gamma$, $\Delta$.
\[
\infer[\shef R]{\Gamma \vdash
  \Delta, A \shef B}{A, B, \Gamma \vdash \Delta} 
 \qquad 
\infer[\shef L]{A \shef B, \Gamma \vdash \Delta}{\Gamma \vdash \Delta,
  A & \Gamma \vdash \Delta, B}
\]
Let us call the sequent calculus that includes these two rules, the
axioms $A \vdash A$, and the usual structural rules (weakening,
exchange, contraction on the left and right, as well as cut)~$\LS$. 

\begin{proposition}
$\LS$ is sound and cut-free complete.
\end{proposition}

\begin{proof}
Follows as a special case of \citet[Theorems~3.1,
  3.2]{BaazFermZach:94} and \citet[Theorems~3.3.3, 3.3.10]{Zach:93}.
\qed\end{proof}

Gentzen-style natural deduction rules are obtained from sequent
calculus rules by turning the premises ``sideways.'' Formulas in the
antecedent of a premise become assumptions. The auxiliary formulas in
the antecedent are those that may be discharged in an application of
the rule; the context formulas~$\Gamma$ are suppressed. Formulas in
the consequent become the conclusion of the corresponding premise. For
the introduction rules, the premises then are just the sideways
premises of the $\shef$R rule.  For the elimination rule, the
principal formula $A \shef B$ in the $\shef L$ rule becomes an
additional premise for the $\shef E$ rule. The natural deduction rules
for~$\shef$ obtained this way then are:
\[
\infer[\shef I_m]{\Delta, A \shef B}{\infer*{\Delta}{[A] & [B]}}
\qquad
\infer[\shef E_m]{\Delta}{\Delta, A \shef B & \Delta, A & \Delta, B}
\]
Note that this version of natural deduction allows multiple formulas
in the conclusion position.  Let us refer to this system as $\NS_m$,
the subscript~$m$ for ``multiple conclusion.''  It requires structural
rules that allow conclusions to be weakened, contracted, and
rearranged.  As for all the natural deduction systems constructed in
this way, there are simple translations of proofs in $\LS$ into
deductions in $\NS_m$ and vice versa. The system $\NS_m$ is thus also
sound and complete for the classical Sheffer stroke.

\begin{proposition}
$\LS$ proves a sequent $\Gamma \vdash \Delta$ iff $\NS_m$ derives
  $\Delta$ from assumptions~$\Gamma'$ with $\Gamma' \subseteq \Gamma$.
\end{proposition}

\begin{proof}
Follows from the general results in \citet[Theorem 4.7,
  5.4]{BaazFermZach:93} and \citet[Theorem~4.2.8, 4.3.4]{Zach:93}. See
Prop.~\ref{prop:transl} below for the translation in this specific
case.
\qed\end{proof}

\section{Single-conclusion natural deduction}

Natural deduction systems commonly require that each conclusion is a
single formula.  In order to obtain rules for a single-conclusion
natural deduction system, we begin by introducing a corresponding
restriction into the sequent calculus: the succedent in each sequent
in a proof must contain at most one formula.  This in effect means
that the $\shef L$, $\shef R$, and $WR$ (weakening-right) rules
become:
\[
\infer[\shef R']{\Gamma \vdash
  A \shef B}{A, B, \Gamma \vdash } 
 \qquad 
\infer[\shef L']{A \shef B, \Gamma \vdash }{\Gamma \vdash A & \Gamma \vdash B}
\qquad
\infer[WR']{\Gamma \vdash A}{\Gamma \vdash}
\]
Exchange and contraction on the right are no longer needed. We will
refer to the resulting calculus as~$\LS'$.

For the corresponding natural deduction system, we also have to
accommodate the empty sequence of conclusion formulas, i.e., we have
to introduce a placeholder symbol~$\bot$ for the empty conclusion in
proofs.  The inference rules will be adapted for single conclusions by
replacing empty conclusions sequences by~$\bot$.  The single
conclusion rules corresponding to $\shef R'$, $\shef L'$, and $WR'$
then are:
\[
\infer[\shef I]{A \shef B}{\infer*{\bot}{[A] & [B]}}
\qquad
\infer[\shef E]{\bot}{A \shef B & A & B}
\qquad
\infer[\bot_I]{A}{\bot}
\]
The $\bot_I$ rule is named as in \citep{Prawitz:65}; \citet{Gentzen:34}
did not give a name to the rule.  The special cases where $A = B$
correspond to Gentzen's $\lnot I$ and $\lnot E$ rules:
\[
\infer[\lnot I]{\lnot A}{\infer*{\bot}{A}}
\qquad
\infer[\shef I]{A \shef A}{\infer*{\bot}{A}} 
\qquad\qquad
\infer[\lnot E]{\bot}{\lnot A & A}
\qquad
\infer[\shef E]{\bot}{A \shef A & A}
\]
The resulting system $\NS$ exactly corresponds to $\LS'$, the sequent
calculus where the succedent of each sequent is restricted to at most
one formula.\footnote{The use of the special symbol~$\bot$ is not
  strictly necessary, as it could be replaced by explicitly
  contradictory premises (e.g., $C \shef D$, $C$, and $D$; or more
  simply $C \shef C$ and~$C$) when it appears as a premise, and by an
  arbitrary formula when it appears as the conclusion. This has the
  advantage that the resulting rules mention no symbols other
  than~$\shef$, i.e., are \emph{pure}.  Indeed, this was in part Hazen
  and Pelletier's goal and their rules used this approach. This
  alternative approach can however not be used if no such ``explicit
  contradiction'' can be expressed. See also Section~\ref{fitch}
  below.}

\begin{proposition}\label{prop:transl}
$\NS$ derives $A$ from assumptions in $\Gamma$ iff $\LS'$ proves the
  sequent $\Gamma \vdash A$ (or $\Gamma \vdash$ if $A = \bot$).
\end{proposition}

\begin{proof}
We translate each proof in $\LS'$ of $\Gamma \vdash A$
into a deduction in $\NS$ of $A$, and one of $\Gamma
\vdash$ into one of $\bot$, where all open assumptions are
in~$\Gamma$. Conversely, any deduction in $\NS$ of $A$ from
assumptions~$\Gamma$ can be translated into a proof in $\LS'$ of
$\Gamma \vdash A$, or of $\Gamma \vdash$ if $A = \bot$.

Assume $\LS$ proves $\Gamma \vdash A$ or (or $\Gamma \vdash$) via a
proof~$\pi$. We construct an $\NS$ derivation $\delta$ of $A$ (or
$\bot$) from assumptions $\Gamma' \subseteq \Gamma$ by induction. If
$\pi$ only consists of the an axiom $A \vdash A$, then $A$ is a
derivation of $A$ (which depends on~$A$).  If $\pi$ ends in a $WR$,
the conclusion is $\Gamma \vdash A$ and the premise is $\Gamma
\vdash$.  By induction hypothesis, we have a derivation of $\bot$ from
$\Gamma'$.  Add $\bot_I$ with conclusion $A$ to this derivation.

If $\pi$ ends in a cut, the last inference is of the form
\[
\infer[cut]{\Gamma_1, \Gamma_2 \vdash B}{\Gamma_1
  \vdash A &A, \Gamma_2 \vdash B}
\]
By induction hypothesis, we have $\NS$ derivations $\delta_1$ of $A$
from assumptions in $\Gamma_1$ and a derivation of $B$ from
assumptions in $A$, $\Gamma_2$.

If $A$ is not actually an open assumption in $\delta_2$, $\delta_2$
already is a derivation of $B$ from $\Gamma_2' \subseteq \Gamma_1 \cup
\Gamma_2$.  Otherwise, $\delta_2$ does contain undischarged
assumptions~$A$.  Append the derivation $\delta_1$ of $A$ (with open
assumptions $\Gamma_1' \subseteq \Gamma_1$) to all leaves where $A$ is
undischarged in $\delta_2$. This is a derivation of~$B$ from
assumptions $\Gamma_1', \Gamma_2' \subseteq \Gamma_1, \Gamma_s$.

If $\pi$ ends in $\shef R'$, the last sequent is $\Gamma \vdash A
\shef B$ and the premise is $A, B, \Gamma \vdash$. By induction
hypothesis there is a derivation $\delta_1$ of $\bot$ with open
assumptions among $A, B, \Gamma$. Add a $\shef I$ rule to $\delta_1$
and discharge any assumptions of the form $A$ or $B$.

If $\pi$ ends in $\shef L$, the last sequent is $A \shef B, \Gamma
\vdash $ and the premises are $\Gamma \vdash A$ and $\Gamma \vdash
B$. By induction hypothesis we have deductions $\delta_1$ of $A$ from
$\Gamma_1 \subseteq \Gamma$ and $\delta_2$ of $B$ from $\Gamma_2
\subseteq \Gamma$, respectively. Use these and an assumption $A \shef
B$ as premises for a $\shef E$ rule with conclusion~$\bot$. The
result is a derivation of $\bot$ from assumptions $A \shef B,
\Gamma_1, \Gamma_2 \subseteq A \shef B, \Gamma$.

In the other direction we proceed the same way. Derivations consisting
of assumptions~$A$ alone are translated into axiom sequents $A \vdash A$.
Derivations ending in $\bot_I$ result in proofs ending in $WR$.  If a
derivation ends in $\shef I$ with the premise $\bot$ being derived
from assumptions $A, B, \Gamma$, by induction hypothesis we have an
$\LS'$-proof of $A, B, \Gamma \vdash$. Apply $\shef R$ to obtain a
proof of $\Gamma \vdash A \shef B$.  If a derivation ends in $\shef
E$, its premises are $A \shef B$, $A$, and $B$ and its conclusion
is~$\bot$. By induction hypothesis, we have $\NS'$-proofs $\pi_1$,
$\pi_2$, $\pi_3$ of $\Gamma_1 \vdash A \shef B$, $\Gamma_2 \vdash A$
and $\Gamma_3 \vdash B$ with $\Gamma_i \cup \Gamma_2 \cup \Gamma_3 =
\Gamma$.  We obtain a proof of $\Gamma \vdash \Delta$ as follows:
\[
  \infer[cut]{\Gamma_1, \Gamma_2, \Gamma_3 \vdash }{
    \infer*[\pi_1]{\Gamma_1 \vdash A \shef B}{} & 
    \infer[\shef L]{A \shef B, \Gamma_2, \Gamma_3 \vdash }{
      \infer={\Gamma_2, \Gamma_3 \vdash A}{
        \infer*[\pi_2]{\Gamma_2 \vdash A}{}} &
      \infer={\Gamma_2, \Gamma_3 \vdash B}{
        \infer*[\pi_3]{\Gamma_3 \vdash B}{}}}}
\]
\qed\end{proof}

\section{Intuitionistic Sheffer strokes}
\label{int}

It is easily seen that the systems $\LS'$ and $\NS$ are not complete.
Consider a three-valued interpretation with truth values $\bot < I <
\top$ with $\top$ designated, and consider the following truth table
for $\shef$:
\[
\begin{array}{c|ccc}
A \shef B & \bot & I & \top \\
\hline
\bot & \top & \top & \top \\
I & \top & \bot  & \bot\\
\top & \top & \bot & \bot
\end{array}
\]
$\LS'$ is sound for this interpretation, i.e., whenever $\LS'$ proves
$\Gamma \vdash A$ then $\min \{ v(A) \colon A \in \Gamma\} \le v(A)$,
and whenever it proves $\Gamma \vdash$, then $\min\{ v(A) : A \in
\Gamma\} = \bot$, for any truth-value assignment~$v$. Consider
$(A\shef A)\shef(A\shef A)$. For the assignment $v(A) = I$, we have
$v((A\shef A)\shef(A\shef A)) = \top > v(A)$, hence the sequent $(A
\shef A)\shef (A \shef A) \vdash A$ is unprovable in
$\LS'$. Consequently there is no $\NS$ deduction of $A$ from the only
assumption~$(A \shef A)\shef (A \shef A)$.  This should not be
surprising. If we abbreviate $A \shef A$ as $\lnot A$, then this
amounts to the intuitionistically invalid inference of double negation
elimination, i.e., $\lnot\lnot A \vdash A$.

Since $\LS'$ and $\NS$ are not complete for the Sheffer stroke the
question arises if one can give a semantics for~$\shef$ relative to
which these systems are complete.  Indeed this is the case. It is
again not surprising that a version of intuitionistic Kripke frames
can serve this purpose.  This can be obtained by applying a
general result about all finite-valued logics due to
\citet{BaazFermuller:96} to the two-valued case and the truth table
for~$\shef$.  An intuitionistic Kripke model is $M = \langle W, \le,
v\rangle$ where $W$ is a non-empty set of worlds, $\le$ is a partial
order on~$W$, and $v$ is a function from the propositional variables
to $\wp(W)$ satisfying $p \in v(w')$ whenever $w \le w'$ and $p \in
v(w)$ (in other words, the interpretation of predicate variables is
monotonic).  We then define the satisfaction relation by:
\begin{enumerate}
\item $M, w \Vdash p$ iff $v(p) \in w$
\item $M, w \Vdash A \shef B$ iff for all $w' \ge w$, either $M, w'
  \nVdash A$ or $M, w' \nVdash B$.
\end{enumerate}
(If one wants to also include other connectives $\lnot$, $\lor$,
$\land$, etc., the corresponding standard satisfaction conditions for
them may be added as well. It is then evident that $M, w \Vdash A
\shef B$ iff $M, w \Vdash \lnot(A \land B)$ as one would expect.)  One
can then easily verify soundness, and prove the following completeness
theorem by a version of Sch\"utte's reduction tree method:

\begin{proposition}
If $\LS'$ does not prove $\Gamma \vdash A$, then there is a Kripke
model $M$ and a world~$w$ so that $M, w \Vdash B$ for all $B \in
\Gamma$ but $M, w \nVdash A$.
\end{proposition}

\begin{proof}
See \citet[Theorem~4.2]{BaazFermuller:96}.
\qed\end{proof}

If the rules $\shef I$, $\shef E$ are added to $\NJ$ instead of being
considered on their own, we can also show completeness by proving that
$A \shef B \vdash \lnot(A \land B)$ and $\lnot(A \land B) \vdash A
\shef B$.  This is what Hazen and Pelletier call ``parasitic
completeness,'' but it has one interesting consequence: it shows that
our rules for $\shef$ in $\LS$ (and hence also the intelim rules in
$\NS$) are---unlike their rules---not ambiguous between two different
connectives: in $\LS$ and $\NS$, $\shef$ defines the intuitionistic
connective $\lnot(A \land B)$.

This of course raises the question: what \emph{do} Hazen and
Pelletier's rules define, if anything?

Let us use $\fesh$ for the intuitionistic connective introduced by Hazen
and Pelletier.  Their rules are formulated for a Ja\'skowski-Fitch
style natural deduction system in which proofs (and subproofs) are
sequences of formulas with scope of assumptions indicated, and instead
of a primitive~$\bot$ they make use of ``explicit contradictions,''
i.e., triples of formulas of the forms~$A \shef B$, $A$, $B$. In our
Gentzen-style framework, their rules can be formulated as follows:
\[
\infer[\fesh I_1]{A \fesh B}{\infer*{\bot}{[A]}} \qquad 
\infer[\fesh I_2]{A \fesh B}{\infer*{\bot}{[B]}} \qquad 
\infer[\fesh E]{C}{A \fesh B & A & B}
\]
These rules correspond to the intuitionistic sequent rules:
\[
\infer[\fesh R'_1]{\Gamma \vdash A \fesh B}{A, \Gamma \vdash } 
 \qquad 
\infer[\fesh R'_2]{\Gamma \vdash A \fesh B}{B, \Gamma \vdash } 
 \qquad 
\infer[\fesh L']{A \fesh B, \Gamma \vdash }{\Gamma \vdash A & \Gamma \vdash B}
\]
The versions $\fesh R_1$, $\fesh R_2$, $\fesh L$ of these rules
without the restriction to a single formula in the succedent are of
course equivalent to $\shef L$ and $\shef R$, i.e., in the presence of
contraction on the right, $\LS$ plus $\fesh R_1$, $\fesh R_2$, $\fesh
L$ proves both $(A \shef B) \vdash (A \fesh B)$ and $(A \fesh B)
\vdash (A \shef B)$.  In the intuitionistic system, i.e., $\LS'$ plus
$\fesh R_1'$, $\fesh R_2'$, $\fesh_L'$, only $(A \fesh B) \vdash (A
\shef B)$ is provable.  This corresponds to the provability in $\LJ$
of $(\lnot A \lor \lnot B) \vdash \lnot(A \land B)$ and the
\emph{un}provability of $\lnot(A \land B) \vdash (\lnot A \lor \lnot
B)$.

However, while $\shef$ defined by $\shef R'$, $\shef L'$ characterizes
$A \shef B$ intuitionistically as $\lnot(A \land B)$, $\fesh R_1'$,
$\fesh R_2'$, $\fesh L'$ do \emph{not}, as one might at first glance
suspect, intuitionistically define $A \fesh B$ as $\lnot A \lor \lnot
B$.  In other words, $\fesh R_1'$, $\fesh R_2'$, $\fesh L'$ is not
parasitically complete: $\LJ$ plus these rules does not prove $(A
\fesh B) \vdash (\lnot A \lor \lnot B)$. (This can be shown by
appealing to \possessivecite{Harrop:60} extended version of the
disjunction property in~$\LJ$; see \citealt[Theorem 6.14]{Takeuti:87}
or \citealt[Theorem 4.2.3]{Troelstra:00}.) Note also that the rules do
not suffice to even prove $A \fesh B \vdash A \fesh B$ from axioms $A
\vdash A$ and $B \vdash B$.

In order to complete the calculus so as to obtain these results, one
could replace the $\fesh L$ rule with the following two:
\[
\infer[\fesh L'']{A \fesh B, \Gamma \vdash C}{A \fesh A, \Gamma \vdash
  C & B\fesh B, \Gamma \vdash C}
\qquad
\infer[\lnot R]{A \fesh A, \Gamma \vdash}{\Gamma \vdash A}
\]
The corresponding natural deduction rules are:\footnote{See errata.}
\[
\infer[\fesh E']{C}{A \fesh B & \infer*{C}{[A]} & \infer*{C}{[B]}}
\qquad
\infer[\lnot E]{\bot}{A\fesh A \qquad A}
\]

\section{Classical single-conclusion systems}

We have seen that the intelim rules for $\shef$ do not by themselves
yield a complete natural deduction system for the classical Sheffer
stroke.  To obtain a classically complete system we have to add a new
rule~$\shef E_C$ which corresponds to the classical $\lnot E$ rule
(Prawitz's $\bot_C$).
\[
\infer[\bot_C]{A}{\infer*{\bot}{\lnot A}}
\qquad
\infer[\shef E_C]{A}{\infer*{\bot}{A \shef A}}
\]
This results in a system~$\NS_C$ complete for the classical
interpretation of~$\shef$. For instance, we can deduce~$A$ from
$(A\shef A)\shef(A \shef A)$ as follows:
\[
\infer[\shef E_C: 1]{A}{
  \infer[\shef E]{\bot}{(A \shef A)\shef(A \shef A) & [A \shef A]^1}}
\]

\begin{proposition}\label{prop:NS-C}
$\NS_C$ is sound and complete.
\end{proposition}

\begin{proof}
Completeness is proved by extending the translation between
$\LS'$-proofs and $\NS$-deductions of Prop~\ref{prop:transl} to
translations between $\LS'_C$-proofs and $\NS_C$-deductions, where
$\LS'_C$ is the calculus resulting from $\LS'$ by adding the rule
\[
\infer[\shef L_C]{\Gamma \vdash A}{\lnot A, \Gamma \vdash A}
\]
One then also shows, by induction on the height of proofs, that if
$\LS'_C$ proves $\Gamma, \lnot \Delta \vdash A$ then $\LS$ proves
$\Gamma \vdash \Delta, A$, and conversely, if $\LS$ proves $\Gamma
\vdash \Delta, \Delta'$ then $\LS'_C$ proves $\Gamma, \lnot \Delta
\vdash \Delta'$, where $\Delta'$ may be empty or a single
formula. This establishes that $\LS'_C$ is complete. (See also
\citealt[Section~3.3.2]{Troelstra:00}; our rule $\shef L_C$ plays the
role of stability axioms here.)
\qed\end{proof}

\section{Normalization}

One important fact about Gentzen's natural deduction calculus $\NJ$ is
that it \emph{normalizes}, i.e., every derivation can be transformed
into one in which is normal: no formula occurrence is both the
conclusion of an introduction rule and the major premise of an
elimination rule. Such formula occurrences are called \emph{maximal
  formulas}. This result was first established in print by
\cite{Prawitz:65}. It is proved by showing that maximal formula
occurrences can be eliminated from derivations, one at a time. The
result then follows by induction on the number of maximal formula
occurrences in a derivation. 

The result also holds for $\NS$, and indeed also for $\NJ$ to which we
add the $\shef I$ and $\shef E$ rules.  All that is required is to
provide the \emph{reduction} for maximal formulas of the form~$A \shef
B$; the argument of \citet[Ch.~II]{Prawitz:65} otherwise goes through
without change.  If a formula occurrence of the form $A \shef B$ is
maximal in a derivation, the derivation has the form on the left, and
can be transformed into the derivation on the right, in which the
maximal formula occurrence is removed:
\[
\infer[\shef E]{\bot}{
  \infer[\shef I]{A \shef B}{\infer*[\delta_1]{\bot}{[A] & [B]}} & 
  \infer*[\delta_2]{A}{} & 
  \infer*[\delta_3]{B}{}}
\qquad
\infer*[{\delta_1[\delta_2/A,\delta_3/B]}]{\bot}{
  \infer*[\delta_2]{A}{} & \infer*[\delta_3]{B}{}}
\]
The derivation $\delta_1[\delta_2/A,\delta_3/B]$ is the derivation
resulting from $\delta_1$ by appending $\delta_2$ to all assumptions
of the form $A$ discharged by the $\shef I$ rule, and $\delta_3$ to
all such assumptions of the form~$B$.

\citet[Ch.~III]{Prawitz:65} also showed that the classical system with
intelim rules for $\lnot$, $\land$, $\to$, $\forall$, with $\bot_C$,
but without $\lor$ and $\exists$ also has normalization. Prawitz shows
that for this system, we may restrict applications of $\bot_C$ to
atomic formulas~$A$.  This is done by showing that any application of
$\bot_C$ for a composite formula can be transformed into one where
$\bot_C$ is instead applied to its immediate subformulas.  The result
then follows by induction on the complexity of formulas~$A$ used
in~$\bot_C$ inferences. It is this step that does not work for
formulas of the form $A \lor B$ or $\exists x\, A(x)$.

We may extend Prawitz's result to $\NS_C$ and to natural deduction
systems which include $\lnot$ and its rules including $\bot_C$
directly, or which define $\lnot A$ as $A \shef A$ and use the rules
for $\shef$, including~$\shef E_C$.  To do this, it suffices to show
that the fact mentioned above (i.e., Prawitz's Th.~1 of Ch.~III) also
holds for $\NS$ and related systems not including~$\lor$.  An
application of $\shef E_C$ applied to $A \shef B$ in a derivation
appears as the subderivation on the left, which can be replaced by the
subderivation on the right (for simplicity, abbreviate $C \shef C$ by
$\lnot C$):
\[
\infer[\shef E_C: 1]{A \shef B}{
  \infer*[\delta]{\bot}{[\lnot (A \shef B)]^1}}
\quad
\infer[\shef I: 2]{A \shef B}{
  \infer*[\delta']{\bot}{
    \infer[\shef I: 1]{\lnot(A \shef B)}{
      \infer[\shef E]{\bot}{[A \shef B]^1 & [A]^2 & [B]^2}}}}
\]
In $\delta'$, the assumptions in $\delta$ of the form $\lnot(A \shef
B)$ which are discharged by the~$\shef E_C$ rule are replaced by the
derivations of $\lnot(A \shef B)$ indicated.

\section{Ja\'skowski-Fitch and Suppes-Lemmon natural deduction}
\label{fitch}

The natural deduction systems of \cite{Jaskowski:34} and
\cite{Fitch:52} and their variants are not commonly used for
proof-theoretic investigations, but they are the most common systems
of natural deduction appearing in introductory textbooks.  They differ
from Gentzen-style single-conclusion natural deduction systems in a
number of ways: First, while in Gentzen-type systems proofs have a
tree structure, in Ja\'skowski-Fitch style systems they are linear
sequences of formulas, which may however be grouped into possibly
nested \emph{subproofs}.  Every subproof begins with an assumption,
and subproofs themselves can serve as premises to inferences. The
conclusion of these inferences then no longer depend on the assumption
of the subproof.

In order to obtain a Ja\'skowski-Fitch type system for the Sheffer
stroke, specifically, to accommodate the $\shef I$ rule, we must allow
for subproofs to have two assumption formulas, or alternatively, relax
the restriction on subproofs having their last formula not appear in a
nested subproof itself.  The corresponding versions of the $\shef I$,
$\shef E$, and $\shef E_C$ rules are as follows:\footnote{The $\shef
  E$ rules below do not correspond to the natural deduction rules
  given above. See errata section at end.}
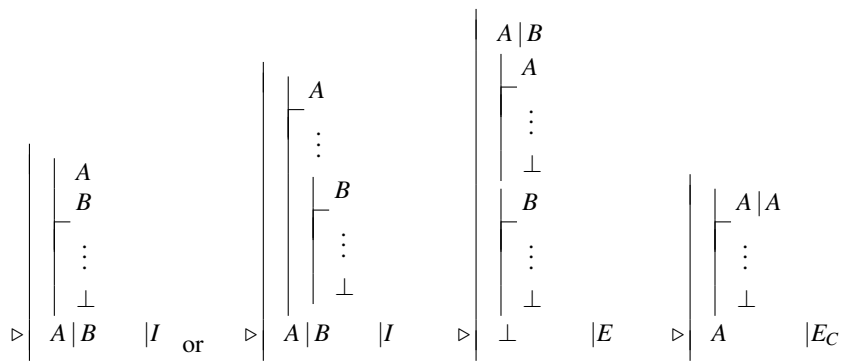
\[
\fitchctx{
\subproof{\nline{A}\\
\nline{B}}{
\ellipsesline\\
\pline{\bot}}
\fpline{A \shef B}[$\shef I$]
}
\mathrm{or} \quad
\fitchctx{
\subproof{\pline{A}}{
\ellipsesline\\
\subproofx{\pline{B}}{
\ellipsesline\\
\pline{\bot}}}
\fpline{A \shef B}[$\shef I$]
}
\quad
\fitchctx{
\pline{A \shef B}\\
\subproof{\pline{A}}{
\ellipsesline\\
\pline{\bot}}
\subproof{\pline{B}}{
\ellipsesline\\
\pline{\bot}}
\fpline{\bot}[$\shef E$]}
\quad
\fitchctx{
\subproof{\pline{A\shef A}}{
\ellipsesline\\
\pline{\bot}}
\fpline{A}[$\shef E_C$]}
\]
Some presentations of Fitch-style systems lack the $\bot$ constant as
a primitive, and instead require the presence of an ``explicit
contradiction,'' i.e., both $A$ and $\lnot A$, when $\bot$ appears as a
premise or the last line of a subproof used as a premise.  Instead of
deriving $\bot$ as a conclusion, the corresponding rules allow the
derivation of an arbitrary formula. Rules which result in such a
system (or which can be added to such a system) would be:
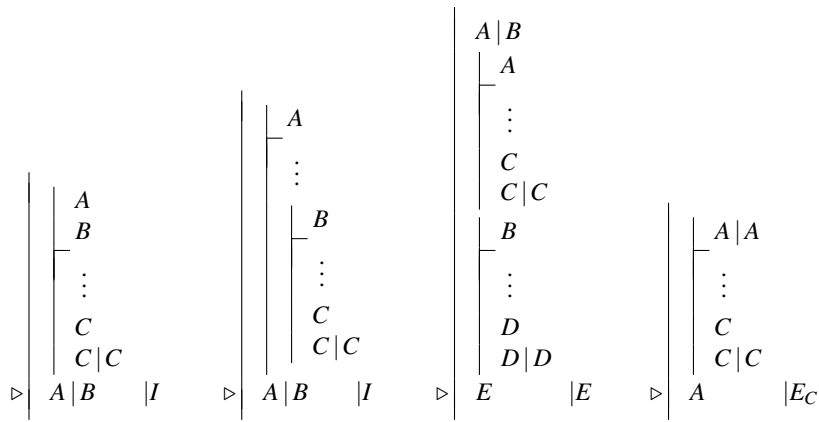
\[
\fitchctx{
\subproof{\nline{A}\\
\nline{B}}{
\ellipsesline\\
\pline{C}\\
\pline{C\shef C}}
\fpline{A \shef B}[$\shef I$]
}
\quad 
\fitchctx{
\subproof{\pline{A}}{
\ellipsesline\\
\subproofx{\pline{B}}{
\ellipsesline\\
\pline{C}\\
\pline{C \shef C}}}
\fpline{A \shef B}[$\shef I$]
}
\quad
\fitchctx{
\pline{A \shef B}\\
\subproof{\pline{A}}{
\ellipsesline\\
\pline{C}\\
\pline{C \shef C}}
\subproof{\pline{B}}{
\ellipsesline\\
\pline{D}\\
\pline{D \shef D}}
\fpline{E}[$\shef E$]}
\quad
\fitchctx{
\subproof{\pline{A\shef A}}{
\ellipsesline\\
\pline{C}\\
\pline{C \shef C}}
\fpline{A}[$\shef E_C$]}
\]

No relaxation of the syntax is needed in the style of natural
deduction due to \cite{Suppes:57} and \citealt{Lemmon:65}, where the
assumptions on which a formula depends are recorded in the deduction
with the formulas themselves, and discharging of assumptions is done
by allowing formulas to be removed from the assumption list.  The
rules might be stated as follows, following Lemmon:
\begin{quote}
$\shef$-Introduction: Given a proof of $\bot$ from $A$ and $B$ as
  assumptions, we may derive $A \shef B$ as conclusion. The conclusion
  depends on any assumptions on which $\bot$ depends in its derivation
  from $A$ and $B$ (apart from $A$ and $B$).

$\shef$-Elimination: Given a proof of $A \shef B$, together with a
  proof of $\bot$ from $A$ as assumption and a proof of $\bot$ from
  $B$ as assumption, we may derive $\bot$ as conclusion. The
  conclusion $\bot$ depends on any assumptions on which $A \shef B$
  depends, or on which $\bot$ depends in its derivation from~$A$
  (apart from~$A$), or on which $\bot$ depends in its derivation
  from~$B$ (apart from~$BA$).

Classical $\shef$-Elimination: Given a proof of $\bot$ from $A \shef
A$ as assumption, we may derive $A$ as conclusion. $A$ depends on any
assumptions on which $\bot$ depends in its derivation from $A \shef A$
(apart from $A \shef A$).
\end{quote}

In these rules we may consider $\bot$ primitive, or read it as ``a
contradiction,'' which again might mean ``both $C$ and $C \shef C$''
if no negation is available.

\section{Other connectives}

It bears emphasizing that the methods by which we have obtained the
rules for~$\shef$ are completely general and will produce rules for
any other truth-functional connective with the same systematic
features.  For instance, Pierce's arrow~$\pier$ or \textsc{nor}, the
dual of the Sheffer stroke, would be characterized by
\begin{description}
\item[T] $A \pier B$ is true iff $A$ is false and $B$ is false
\item[F] $A \pier B$ is false iff $A$ is true or $B$ is true
\end{description}
which would result in the sequent rules
\[
\infer[\pier R]{\Gamma \vdash \Delta, A \pier B}{A, \Gamma \vdash
  \Delta & B, \Gamma \vdash \Delta} 
\qquad 
\infer[\pier L]{A \pier B, \Gamma
  \vdash \Delta}{\Gamma \vdash \Delta, A, B} \qquad
\]
In order to obtain a single-conclusion sequent system we again leave
out the side formulas~$\Delta$. The presence of two formulas on the
right in the premise of $\pier$L is an obstacle for turning this rule
into single-conclusion natural deduction rule. However, we can replace
$\pier L$ by two rules in which the auxiliary formulas $A$ and $B$
appear in the premise only once.  By restricting the succedents to at
most one formula, we obtain the following rules:
\[
\infer[\pier R]{\Gamma \vdash A \pier B}{A, \Gamma \vdash
  & B, \Gamma \vdash } 
\qquad 
\infer[\pier L_1]{A \pier B, \Gamma
  \vdash }{\Gamma \vdash A} \qquad
\infer[\pier L_2]{A \pier B, \Gamma
  \vdash }{\Gamma \vdash B}
\]
We obtain the single-conclusion natural deduction rules
\[
\infer[\pier I]{A \pier B}{\infer*{\bot}{[A]} & \infer*{\bot}{[B]}}
\qquad
\infer[\pier E_1]{\bot}{A \pier B & A}
\qquad
\infer[\pier E_2]{\bot}{A \pier B & B}
\]
This natural deduction system, as well as the restricted sequent
calculus, are again not classically complete, but are complete for an
intuitionistic version of~$\pier$.  They can be made complete for
classical~$\pier$ by adding
\[
\infer[\pier E_C]{A}{\infer*{\bot}{A \pier A}}
\]

A slightly more complex example would be the \textsc{xor} connective.
Its truth and falsity conditions are given by:
\begin{description}
\item[T] $A \xor B$ is true iff ($A$ is true or $B$ is true) and ($A$
  is false or $B$ is false)
\item[F] $A \xor B$ is false iff ($A$ is false or $B$ is true) and ($B$
  is false or $A$ is true)
\end{description}
The corresponding sequent rules are
\[
\begin{array}{l}
\infer[\xor R]{\Gamma \vdash \Delta, A \xor B}{\Gamma \vdash
  \Delta, A, B & A, B, \Gamma \vdash \Delta} 
\qquad
\infer[\xor L]{A \xor B, \Gamma
  \vdash \Delta.}{A, \Gamma \vdash \Delta, B & B, \Gamma \vdash \Delta, A}
\end{array}
\]
Again, the $\xor$R rule may be replaced by
the two rules
\[
\begin{array}{l}
\infer[\xor R_1]{\Gamma \vdash \Delta, A \xor B}{\Gamma \vdash
  \Delta, A & A, B, \Gamma \vdash \Delta} 
\qquad
\infer[\xor R_2]{\Gamma \vdash \Delta, A \xor B}{\Gamma \vdash
  \Delta, B & A, B, \Gamma \vdash \Delta} 
\end{array}
\]
which result in the natural deduction rules
\[
\infer[\xor I_1]{A \xor B}{A & \infer*{\bot}{[A], [B]}}
\qquad
\infer[\xor I_2]{A \xor B}{B & \infer*{\bot}{[A], [B]}}
\qquad
\infer[\xor E]{\bot}{A \xor B & \infer*{B}{[A]} & \infer*{A}{[B]}}
\]
In this case, $\xor$ cannot be used to express~$\lnot$. It is thus not
clear whether a complete system can be obtained without adding $\lnot$
as a primitive together with the full classical set of rules for it.

Finally, the same method can be applied to
\possessivecite{Schonfinkel:24} quantifier $\shef^x$, the basis of his
$U$ combinator.  $A(x) \shef^x B(x)$ is true iff for every $x$, not
both $A(x)$ and $B(x)$ hold. Sound and complete sequent calculus rules
are
\[
\infer[\shef^x R]{\Gamma \vdash \Delta, A(x) \shef^x B(x)}{A(a),
  B(a), \Gamma \vdash \Delta}
\qquad
\infer[\shef^x L]{A(x) \shef^x B(x), \Gamma \vdash \Delta}{\Gamma
  \vdash \Delta, A(t) & \Gamma \vdash \Delta, B(t)}
\]
which lead to the following intelim rules,
\[
\infer[\shef^x I]{A(x) \shef^x B(x)}{\infer*{\bot}{[A(a)] &
 [B(a)]}}
\qquad
\infer[\shef^x E]{\bot}{A(a) \shef^x B(a) & A(t) & B(t)}
\]
all of course with the usual eigenvariable conditions on~$a$.

\section{Conclusion}

The specific question of how to come up with left and right rules for
the Sheffer stroke in the style of $\LK$, or introduction and
elimination rules in the style of $\NK$ might seem like a mere
recreational logic puzzle.  Nevertheless, the considerations above
have highlighted some interesting aspects of sequent calculus and
natural deduction.  The first is that there are general methods which
generate such rules automatically from their truth tables.  Not only
are the resulting systems complete, they also enjoy the same
properties (like cut-elimination) as the standard systems, and
transformations from one system to another (such as the translations
between proofs and derivations sketched in
Proposition~\ref{prop:transl}) work in both multiple-conclusion
(classical) and single-conclusion (intuitionistic) systems.  Our
consideration of how to make a single-conclusion natural deduction
calculus classical revealed the importance of the $\bot_C$ rule in
this respect: $\lnot L$ allows us to move formulas from the succedent
of a sequent to the right (where they then appear negated), and and
$\lnot E$ allows us in effect to turn a derivation \emph{of}~$A$ into
one of $\bot$ \emph{from}~$\lnot A$.  The rules $\shef L_C$ (of the
proof of Proposition~\ref{prop:NS-C}) and $\bot_C$ enable us to
reverse this, i.e., to move negated formulas from the antecedent of a
sequent to the succedent, and to transform a derivation in which
$\lnot A$ is an open assumption into one with $A$ instead as the
conclusion.  This means that the restriction to a single formula on
the right, which corresponds to the restriction to single conclusions
in natural deduction, can be circumvented. The same trick can be used
whenever the new connective(s) define(s)~$\lnot$ (as in the case of
$\shef$), but it can also be used to show how a pair of intelim rules
can be made classically complete by adding rules for $\lnot$ including
the $\bot_C$~rule.

\begin{acknowledgements}
I am grateful to Allen Hazen and Jeff Pelletier for helpful comments
as well as the question which originally prompted this paper.
\end{acknowledgements}

\bibliographystyle{spbasic} 

\section{Errata}

The $\shef E$ rule of section~\ref{fitch} should read:
\[
\fitchctx{
\pline{A \shef B}\\
\pline{A}\\
\pline{B}\\
\fpline{\bot}[$\shef E$]}
\]
The rule printed in the text corresponds instead to a special version
of the $\fesh E'$ rule of section~\ref{int}. However, that rule is also incorrect. The $\fesh E'$ rule corresponding to $\fesh L''$ is
\[
\infer[\fesh E']{C}{A \fesh B & \infer*{C}{[A \fesh A]} & \infer*{C}{[B \fesh B]}}
\]

\end{document}